\documentclass[11pt,twoside,openbib]{article}
\usepackage{latexsym,amsfonts,amssymb,amsmath,amscd}
\usepackage[T1]{fontenc}
\usepackage{amssymb}
\usepackage{amsthm}
\usepackage{pst-node}
\usepackage{pstricks}
\usepackage{graphicx}
\usepackage{xy}
\usepackage{amsmath}
\usepackage[latin1]{inputenc}
\usepackage[english]{babel}
\usepackage{epsfig}
\usepackage{wasysym}
\usepackage{amsfonts}
\usepackage{multicol}
\usepackage{graphicx}
\usepackage{wasysym}
\usepackage{color}

\usepackage{makeidx}

\makeindex 
\newtheorem{thm}{\bf Theorem}[section]

\newtheorem{lem}[thm]{Lemma}

\newtheorem{rem}[thm]{Remark}
\numberwithin{equation}{section}
\newenvironment{preuve}
{\medskip\noindent{\bf Proof.\\ }}{\\ \null \hfill
$\blacksquare$\par\medskip}

\newcommand{\zero}{\textup{\tiny 0}}
\newcommand{\one}{\textup{\tiny 1}}

\newcommand{\two}{\textup{\tiny 2}}
\newcommand{\mtwo}{\textup{\tiny -2}}
\newcommand{\three}{\textup{\tiny 3}}
\newcommand{\four}{\textup{\tiny 4}}
\newcommand{\five}{\textup{\tiny 5}}
\newcommand{\six}{\textup{\tiny 6}}
\newcommand{\seven}{\textup{\tiny 7}}
\newcommand{\eight}{\textup{\tiny 8}}
\newcommand{\nine}{\textup{\tiny 9}}
\newcommand{\ten}{\textup{\tiny 10}}
\newcommand{\eleven}{\textup{\tiny 11}}
\newcommand{\twelve}{\textup{\tiny 12}}

\newcommand{\Omegapt}{\textup{\tiny $ \Omega$}}

\newcommand{\tr}{\textup{\tiny T}}

\newcommand{\dv}{\textup{div\,}}
\newcommand{\uu}{\mathbf{u}}

\newcommand{\UU}{\mathbf{U}}
\newcommand{\lbf}{\mathbf{{L}}}
\newcommand{\lrm}{\mathrm{{L}}}
\newcommand{\U}{\mathrm{{U}}}
\newcommand{\ff}{\mathbf{f}}
\newcommand{\SSS}{\mathbf{S}}
\newcommand{\ww}{\mathbf{w}}
\newcommand{\vv}{\mathbf{v}}
\newcommand{\nn}{\mathbf{n}}

\newcommand{\FF}{\mathbf{F}}
\newcommand{\HH}{\mathbf{H}}
\newcommand{\HHH}{\mathrm{H}}

\newcommand{\C}{\mathrm{C}}

\newcommand{\Q}{\mathrm{Q}}

\newcommand{\T}{\mathrm{T}}

\newcommand{\we}{\textup{We}}

\newcommand{\ep}{\varepsilon}

\newcommand{\D}{\mathbf{D}}

\newcommand{\DD}{\mathbf{D}}

\newcommand{\dte}[1]{\displaystyle{\frac{\partial#1}{\partial t^*}}}
\newcommand{\dt}[1]{\displaystyle{\frac{\partial#1}{\partial t}}}

\newcommand{\Dt}[1]{\displaystyle{\frac{\mathrm{d}#1}{\mathrm{dt}}}}
\newcommand{\Dae}[1]{\displaystyle{\frac{\mathcal{D}_a#1}{\mathcal{D}t^*}}}

\newcommand{\abs}[1]{\left\vert#1\right\vert}

\newcommand{\n}[1]{\left|\left|#1\right|\right|}

\newcommand{\lebl}[1]{\mathrm{L}^{#1}{(0,\mathrm{T},\mathrm{L}^{\tiny\textup{2}}(\Omega))}}
\newcommand{\leblbf}[1]{\mathrm{L}^{#1}{(0,\mathrm{T},\mathbf{L}^{\tiny\textup{2}}(\Omega))}}
\newcommand{\lebhzero}[1]{\mathrm{L}^{#1}{(0,\mathrm{T},\mathrm{H}_\zero^ {\tiny\textup{1}}(\Omega))}}
\newcommand{\lebhzerobf}[1]{\mathrm{L}^{#1}{(0,\mathrm{T},\mathbf{H}_\zero^ {\tiny\textup{1}}(\Omega))}}
\newcommand{\lebh}[1]{\mathrm{L}^{#1}{(0,\mathrm{T},\mathrm{H}^{\tiny\textup{1}}(\Omega))}}
\newcommand{\lebhbf}[1]{\mathrm{L}^{#1}{(0,\mathrm{T},\mathbf{H}^{\tiny\textup{1}}(\Omega))}}
\newcommand{\lebhh}[1]{\mathrm{L}^{#1}{(0,\mathrm{T},\mathrm{H}^{\tiny\textup{2}}(\Omega))}}
\newcommand{\lebhhbf}[1]{\mathrm{L}^{#1}{(0,\mathrm{T},\mathbf{H}^{\tiny\textup{2}}(\Omega))}}
\newcommand{\lebhhh}[1]{\mathrm{L}^{#1}{(0,\mathrm{T},\mathrm{H}^{\tiny\textup{3}}(\Omega))}}
\newcommand{\lebhhhbf}[1]{\mathrm{L}^{#1}{(0,\mathrm{T},\mathbf{H}^{\tiny\textup{3}}(\Omega))}}

\newcommand{\contllbf}{\mathrm{C}{([0,\mathrm{T}];\mathbf{L}^{\tiny\textup{2}}(\Omega))}}
\newcommand{\conth}{\mathrm{C}{([0,\mathrm{T}];\mathrm{H}^{\tiny\textup{1}}(\Omega))}}
\newcommand{\conthbf}{\mathrm{C}{([0,\mathrm{T}];\mathbf{H}^{\tiny\textup{1}}(\Omega))}}

\newcommand{\conthzerobf}{\mathrm{C}{([0,\mathrm{T}];\mathbf{H}^{\tiny\textup{1}}_\zero(\Omega))}}

\newcommand{\conthhintersection}{\mathrm{C}{([0,\mathrm{T}];\mathbf{H}^{\tiny\textup{2}}(\Omega)\cap\mathbf{H}^{\tiny\textup{1}}_\zero(\Omega))}}

\newcommand{\lebhintersection}[1]{\mathrm{L}^{#1}{(0,\mathrm{T},\mathrm{H}^{\tiny\textup{2}}(\Omega)\cap\mathrm{H}^{\tiny\textup{1}}_\zero(\Omega)) }}
\newcommand{\lebhintersectionbf}[1]{\mathrm{L}^{#1}{(0,\mathrm{T},\mathbf{H}^{\tiny\textup{2}}(\Omega)\cap\mathbf{H}^{\tiny\textup{1}}_\zero(\Omega)) }}
\newcommand{\lebhmoins}[1]{\mathrm{L}^{#1}{(0,\mathrm{T},\mathrm{H}^{\tiny\textup{-1}}(\Omega))}}

\newcommand{\none}[1]{[#1]}

\title{Local existence and uniqueness of solutions for non stationary compressible viscoelastic  fluid  of Oldroyd type}
\begin{document}
\maketitle
\date{}
\begin{center}
\textbf{ Zaynab SALLOUM}\\
 {Universit\'{e} Libanaise, Facult\'{e} des Sciences I,}\\
{D\'{e}partement des Math\'{e}matiques, Hadath, Liban}\\
\verb"salloum@ul.edu.lb"
\end{center}
\begin{abstract}
This work is devoted to the study of a compressible viscoelastic fluids satisfying the Oldroyd-B model in a regular bounded domain. We prove the local existence of solutions and uniqueness of
flows by a classical fixed point argument.
\end{abstract}
\section{Introduction}
In this paper, we study the local  existence of solutions
for   compressible viscoelastic fluid flows in the case of the
Oldroyd-B model in a regular bounded domain in $\mathbb{R}^3.$ We
also show the uniqueness of solutions. We prove l'existence by
using the classical method based on the Schauder fixed-point theorem.  Valli, in \cite{V}, show the local existence in the case of the Navier-Stokes equations. The case of the Oldroyd model for incompressible fluid is studied by Guillopé and Saut in \cite{GS}. Talhouk shows the existence and the  uniqueness  for Jeffreys model's in \cite{Ta2}.

This paper is organized as follows. Section \ref{mod-3} is devoted
to the modeling of the problem and to the definition of
well-prepared initial conditions. The principal notation and results
are detailed in Section \ref{res-3}. The local existence of regular
solutions is given in Section 4.
\section{{The Modeling}}\label{mod-3}
\subsection{Unsteady Flows of Compressible Viscoelastic Fluids}
 Consider
 unsteady flows of viscoelastic fluids in a bounded domain $ \Omega^*$ of $
 \mathbb{R}^\three$ with a regular boundary $\Gamma^*$.
   The system, obtained from the laws of conservation of momentum, and of mass,
and from the constitutive equation of the fluid, reads as follows
\cite{ZS}: in $\Q^*_{\T^*}=(0,T^*)\times \Omega^*$,
 \begin{equation}\label{e1}
\left\{
\begin{array}{r c l l}
\rho^*\left(\dte{\uu^*}+(\uu^*\cdot\nabla^*)\uu^*\right)&=&\rho^*\ff^*+\textbf{\dv}^*(\tau^*-p^*\textbf{I}),\\
\dte{\rho^*}+\dv^*{(\rho^*\uu^*)}&=&0,\\
\tau^*+\lambda \Dae{\tau^*}&=&2\eta \left(\D^*+\mu
\Dae{\D^*}\right).
\end{array}
\right.
\end{equation}
The *-variables  are the dimensional ones in the domain of the flow
$\Omega^*$, and $T^*>0$ is a dimensional time. The unknowns are the
velocities $\uu^*$, the density $\rho^*$,  and
the symmetric tensor of constraints $\tau^*$. $ \eta$ is the total
viscosity of the fluid, $\lambda > 0$ is the relaxation time, and
$\mu$ is the retardation time ($0<\mu<\lambda$).

 $\Dae{\tau^*}$ is an
objective derivative of the tensor $\tau^*$, given by
\begin{equation*}
\Dae{\tau^*}=\left(\frac{\partial}{\partial
t^*}+(\uu^*\cdot\nabla^*)\right)\tau^*+\tau^*\mathbf W^*-\mathbf
W^*\tau^*-a(\D^*\tau^*+\tau^*\D^*),
\end{equation*}
where $\mathbf W^*=\mathbf W^*[\uu^*]=\displaystyle{\frac{1}{2}
(\nabla^* \uu^*-{\nabla^*}^\tr \uu^* )}$ and
$\D^*=\D^*[\uu^*]=\displaystyle{\frac{1}{2}
(\nabla^*\uu^*+{\nabla^*}^\tr \uu^*)}$  are, respectively, the  rate
of rotation and the rate of deformation tensors. $a$ is a real
parameter in $ [-1,1]$.

System (\ref{e1}) is completed by a condition on the boundary,
\begin{equation*}\label{cd}
\uu^*=0 \text{ on } \Sigma^*_{\T^*}=(0,\T^*) \times \Gamma^*,
\end{equation*}
and by the initial data
\begin{equation*}\label{ci}
\uu^*(0,\cdot)=\uu^*_0,\quad \rho^*(0,\cdot)=\rho^*_0,\quad
\tau^*(0,\cdot)=\tau^*_0, \quad \text{in }\Omega^*.
\end{equation*}

We split   $\tau^*$ into two parts: the Newtonian one $\tau^*_s$
related to the solvent, and the polymeric one $\tau^*_p$. We may
write
\begin{equation*}
\tau^*=\tau^*_s+\tau^*_p=2\eta_s\D^*+ \tau^*_e,
\end{equation*}
where
$\tau_e^*=\tau_p^*-\left(\frac{2\xi_s}{3}\dv^*{\uu^*}\right)\textbf{I}$,
and $\textbf{I}$ is the identity tensor. $\eta_s=\eta \mu/\lambda$
and $\xi_s$ are the solvent viscosity and the group viscosity,
respectively. Since we are interested in a model for weakly
compressible fluids, we suppose that $\xi_s=0$.
From the third equation in  $(\ref{e1})$, we can deduce that
$\tau_e^*$ satisfies the equation
\begin{equation*}
\tau^*_e+\lambda \Dae{\tau^*_e}=2\eta_e \D^*,
\end{equation*}
%
%
where $\eta_e=\eta-\eta_s$ is called the polymer viscosity. $\eta_s$
and $\eta_e$ are two non-negative numbers.

Therefore, under the assumption $\xi_s=0$, System (\ref{e1}) is
equivalent to  the system in $\Q^*_{\T^*}$,
\begin{equation}\label{e3}
\left\{
\begin{array}{r c l l}
\rho^*\left(\dte{\uu^*}+(\uu^*\cdot\nabla^*)\uu^*\right)&=&\rho^*\ff^*+\eta_s(\Delta ^*\uu^*+\nabla ^*\dv^*{\uu^*})-\nabla^* p^*+\textbf{\dv}^* \tau^*,\\
\dte{\rho^*}+\dv^*{(\rho^*\uu^*)}&=&0,\\
\tau^*+\lambda \Dae{\tau^*}&=&2\eta_e \D^*[\uu^*],
\end{array}
\right.
\end{equation}
where  we have denoted  $\tau_e^*$ by $\tau^*$ to simplify the
notation.
\subsection{{Well-Prepared Initial Conditions}}
We first define the Mach number $\ep$ as being the ratio of the
typical velocity of the fluid $\U_\zero$ to the speed of sound
$\left(\frac{d p^*}{d \rho^*}(\overline{\rho}^*_\zero)\right)^{1/2}$
in the same fluid at the same state. We divide the density
${\rho}^*={\rho}^{*\ep}$ into two parts: a {\sl constant} one
$\overline{\rho}^*_0$, independent of $\ep$, and a remainder, which
is small for small $\ep's$, say
\begin{equation*}
\rho^{*\ep}=\overline{\rho}^*_\zero+\mathcal{O}(\ep^2)=\overline{\rho}^*_\zero+\ep^2\sigma^{*\ep}.
\end{equation*}
 We also suppose that the initial conditions $\rho^{*\ep}_\zero$,
$\uu_\zero^{*\ep}$ and $\tau^{*\ep}_\zero$ are
\textit{well-prepared}, which means that they take a similar form,
say
\begin{eqnarray*}
\rho^{*\ep}_\zero&=&\overline{\rho}^*_\zero+\mathcal{O}(\ep^2)=\overline{\rho}^*_\zero+\ep^2\sigma^{*\ep}_\zero,\\
\uu^{*\ep}_\zero&=&\vv^{*}_\zero+\vv_\zero^{*\ep}, \textup{ with }
\dv
{\vv^*_\zero}=0,\\
\tau^{*\ep}_\zero&=&\SSS^*_\zero+\SSS^{*\ep}_\zero,
\end{eqnarray*}
where   $\vv^{*}_\zero$ and $\SSS^*_\zero$ are, respectively,
 a vector and a symmetric tensor, both  independent of $\ep$.

  We assume
\begin{equation*}
\mathfrak{m}^*=\min_{\overline{\Omega}^*}{\rho}^*_\zero>0 \quad
\textup{ and } \quad
\mathfrak{M}^*=\max_{\overline{\Omega}^*}{\rho}^*_\zero.
\end{equation*}
 Assuming that $p^*=p^*(\rho^*)$ is regular, say class $\C^\three$ at least, we remark
\begin{equation*}
\frac{d p^*}{d\rho^*}(\overline{\rho}^*_\zero+\ep^2\sigma^*)-\frac{d
p^*}{d\rho^*}(\overline{\rho}^*_\zero)=\ep^2\int_0^1\frac{d^2 p^*}{d
{\rho}^{*2}}(\overline{\rho}^*_\zero+s\ep^2\sigma^*)\,ds.
\end{equation*}
We introduce the function $w^*$, defined by $w^*(\sigma^*)=\frac{d
p^*}{d\rho^*}(\overline{\rho}^*_\zero+\ep^2\sigma^*)-\frac{d
p^*}{d\rho^*}(\overline{\rho}^*_\zero)$, and remark that $w^*$
depends on $\ep$, satisfies $w^*(0)=0$, and is of class $\C^\two$ at
least.

 Replacing $\rho^*$ by its value in
the first equation $(\ref{e3})$, one infers
\begin{eqnarray*}
(\overline{\rho}^*_\zero+\ep^2\sigma^*)\left(\dte{\uu^*}+(\uu^*\cdot\nabla^*)\uu^*\right)+\ep^2\frac{d
p^*}{d\rho^*}(\overline{\rho}^*_\zero+\ep^2\sigma^*)\nabla^*\sigma^*\hskip5cm\\
=(\overline{\rho}^*_\zero+\ep^2\sigma^*)\ff^*+\eta_s(\Delta^*
\uu^*+\nabla^* \dv^*{\uu^*}) +\textbf{\dv}^*\tau^*.\hskip3cm
\end{eqnarray*}
We can also rewrite this equality, by taking into account the
definitions of
 $w^*(\sigma^*)$ and of the Mach number $\ep$, in the form
\begin{eqnarray*}\label{eqt1}
(\overline{\rho}^*_\zero+\ep^\two\sigma^*)\left(\dte{\uu^*}+(\uu^*\cdot\nabla^*)\uu^*\right)+(\UU_\zero)^\two\nabla^*\sigma^*\hskip6cm\nonumber\\
=(\overline{\rho}^*_\zero+\ep^\two\sigma^*)\ff^*+\eta_s(\Delta^*
\uu^*+\nabla^* \dv^*{\uu^*}) +\textbf{\dv}^* \tau^*-\ep^\two
w^*(\sigma^*)\nabla^*\sigma^*.
\end{eqnarray*}
%
%
From the second equation in $(\ref{e3})$ we easily deduce
\begin{equation*}
\ep^\two\dte{\sigma^*}+\overline{\rho}^*_\zero
\dv^*{\uu^*}+\ep^\two\dv^*{(\sigma^*\uu^*)}=0.
\end{equation*}
Finally, System (\ref{e3}) can be written as follows, in
$\Q^*_{\T^*}$,
\begin{equation}\label{e4}
\left\{
\begin{array}{r c l l}
(\overline{\rho}^*_\zero+\ep^\two\sigma^*)\left(\dte{\uu^*}+(\uu^*\cdot\nabla^*)\uu^*\right)+(\UU_\zero)^\two\nabla^*\sigma^*
=  (\overline{\rho}^*_\zero+\ep^\two\sigma^*)\ff^*+\,\textbf{\dv}^*
\tau^*
\\
+\eta_s(\Delta^*
\uu^*+\nabla^* \dv^*{\uu^*})-\ep^\two w^*(\sigma^*)\nabla^*\sigma^*,\\
\ep^\two \dte{\sigma^*}+\overline{\rho}^*_\zero\dv^*{\uu^*}+\ep^\two \dv^*{(\sigma^*\uu^*)} = 0,\hskip3cm\\
\tau^*+\lambda \Dae{\tau^*} = 2\eta_e \D^*[\uu^*].\hskip4cm
\end{array}
\right.
\end{equation}
\subsection{{Dimensionless Variables}} We introduce the dimensionless
variables,
\begin{equation*}
\textbf{x}^*=L_\zero \textbf{x},\quad t^*=\frac{L_\zero}{\U_\zero}
t,\quad \rho^*=a_\zero{\rho},\quad w^*(\sigma^*)=(\U_\zero)^\two
w(\sigma),
\end{equation*}
\begin{equation*}
\uu^* = {\U}_\zero\uu,\quad \sigma^*=a_\zero\sigma ,\quad \tau^*=
T_\zero\tau,\quad p^*(\rho^*)=T_\zero p(\rho) , \quad
\ff^*=\frac{(\U_\zero)^2}{L_\zero} \ff,
\end{equation*}
where $L_\zero$ represents a typical length of the flow. The real
numbers $a_\zero=\frac{\eta}{ \U_\zero L_\zero}$ and
$T_\zero=\frac{\eta \U_\zero}{L_\zero}$ characterize the density and
the stress tensor of the fluid. $\Omega$ denotes the non-dimensional
domain of the flow,  with boundary $\Gamma$,  and $T>0$ a
non-dimensional time.

We introduce three non-dimensional  numbers: a number $\alpha$
similar to the Reynolds number for incompressible flows, the
Weissenberg number $\we$, and a number $\omega$ relative to the
viscosities of the fluid,
\begin{equation*}
\alpha=
\frac{\overline{\rho}^*_\zero}{a_\zero}=\frac{\overline{\rho}^*_\zero
\U_\zero L_\zero}{\eta},\quad \we=\frac{\lambda \U_\zero}{L_\zero},
\quad \omega=1-\frac{\eta_s}{\eta}.
\end{equation*}
We also define
\begin{equation*}\label{defomega} w(\sigma)=\alpha\left\{\frac{dp}{d\rho}\left(\alpha+\ep^\two
\sigma\right)-\frac{dp}{d\rho}\left(\alpha\right)\right\}.\end{equation*}

In dimensionless variables, System (\ref{e4}) takes the form, in
$\Q_\T=(0,T)\times\Omega$,
\begin{equation}\label{e6}
\left\{
\begin{array}{r c l l}
 \uu'+(\uu\cdot\nabla)\uu +\displaystyle\frac{1}{\alpha+\ep^\two\sigma}\nabla\sigma&=&
\ff+\displaystyle\frac{1-\omega}{\alpha+\ep^\two\sigma}(\Delta
\uu+\nabla \dv{\uu})
 +\displaystyle\frac{\textbf{\dv}
\tau}{\alpha+\ep^\two\sigma} \\\\&&-\displaystyle\frac{\ep^\two
w(\sigma)\nabla\sigma}{\alpha+\ep^\two\sigma},\\ \\
\sigma'+\ep^{-2}\alpha\,\dv{\uu}+\dv{(\sigma\uu)}&=&0,\\ \\
\tau+\we\{\tau'+(\uu\cdot\nabla)\tau+\mathbf
g(\nabla\uu,\tau)\}&=&2\omega\D[\uu],
\end{array}
\right.
\end{equation}
with the notation  $ \uu'=\dt{\uu}$, $ \sigma'=\dt{\sigma}$ and
$\tau'=\dt{\tau} $, and
$$
\mathbf g(\nabla \uu,\tau)=\tau \mathbf W[\uu]-\mathbf
W[\uu]\tau-a\Big(\D[\uu]\tau+\tau\D[\uu]\Big).
$$
 Introducing the differential
operator $A=-(\Delta +\nabla \dv)$ we may rewrite System (\ref{e6})
as follows, in $\Q_\T$,
\begin{equation}\label{e7}
\left\{
\begin{array}{r c l l}
\alpha\Bigl({\uu'}+(\uu\cdot\nabla)\uu\Bigl)+(1-\omega)A\uu+\nabla\sigma&=&
\FF(\uu,\sigma,\tau)+\textbf{\dv} \tau,\\ \\
{\sigma'}+(\uu\cdot\nabla)\sigma+\sigma\dv\uu&=&-\ep^{\two}\alpha\,\dv{\uu},\\
\\ \tau+\we\Bigl(\tau'+(\uu\cdot\nabla)\tau+\mathbf g(\nabla\uu,\tau)\Bigl)
&=&2\omega\D[\uu],
\end{array}
\right.
\end{equation}
with
\begin{eqnarray}\label{defF}
\FF(\uu,\sigma,\tau)&=&\alpha\ff+\frac{(1-\omega)\ep^\two\sigma}{\alpha+\ep^\two\sigma}A\uu+
\frac{\ep^\two(\sigma-w(\sigma))}{\alpha+\ep^\two\sigma}\nabla\sigma
-\frac{\ep^\two\sigma}{\alpha+\ep^\two\sigma}\textbf{\dv} \tau.\qquad
\end{eqnarray}
System (\ref{e7}) is completed  by an homogeneous condition on the
boundary,
\begin{equation}\label{cd1}
\uu=0 \text{ on } \Sigma_\T= (0,\T)\times \Gamma,
\end{equation}
and by three initial conditions,
\begin{equation}\label{ci1}
\uu(0,\cdot)=\uu_\zero,\quad \sigma(0,\cdot)=\sigma_\zero,\quad
\tau(0,\cdot)=\tau_\zero,\quad \text{in }\Omega.
\end{equation}

 We also assume the followings,
\begin{equation*}\label{defMm}
0<\mathfrak{m}_\one=\frac{\mathfrak{m}^*}{a_\zero}\leq
\alpha+\ep^\two\sigma_0 \leq
\mathfrak{M}_\one=\frac{\mathfrak{M}^*}{a_\zero}\,, \quad \textup{
in }{\Omega},
\end{equation*}
where $\mathfrak{m}_\one$ and $\mathfrak{M}_\one$ are some given
constants.
\section{{The Notation and Main Results}}\label{res-3}
\subsection{{Notation}}
 $\Omega$
 is a bounded domain in $\mathbb{R}^\three$, with a regular boundary $\Gamma$, and $\nn$ denotes the unit outward-pointing normal
vector to $\Gamma$. For $\textbf{x}=(x^\one,x^\two,x^\three) \in
\,\mathbb{R}^\three$, we denote by $\abs{\textbf{x}}$ its Euclidean
norm.

We will use the following spaces: the Lebesgue spaces
$\lrm^p(\Omega)$, $1\leq p\leq+\infty$, with norms
$\n{\cdot}_{\lrm^p}$
 (except for the $\lrm^2(\Omega)$-norm, which is denoted by $\n{\cdot}$); the Sobolev space $\HHH^k(\Omega)$, $k \in \mathbb{N}^*$,
with norm $\n{\cdot}_k$ and inner product $((\cdot,\cdot))_{k}$; the
vector spaces  $\lbf^2(\Omega)$ and $\HH^k(\Omega)$ of vector-valued
or tensor-valued functions with components in $\lrm^2(\Omega)$ and
$\HHH^k(\Omega)$ respectively, their norms being denoted in the same
way as above. We will also use the homogeneous Sobolev space
$\HHH^1_0(\Omega)$ and its dual $\HHH^{-1}(\Omega)$.

If $I$ is an interval of $\mathbb{R}_+$ and  $k \in \mathbb{N}$,
$\C(\overline{I};\HH^k(\Omega))$ is the space of  vector- or
tensor-valued functions which are  continuous on $\overline{I}$ with
values in $\HH^k(\Omega)$. The norm, in this space, is denoted by
$\n{\cdot}_{\mathcal{C},k}$. $\C_b(\overline{I};\HH^k(\Omega))$ is
the space of functions of $\C({I};\HH^k(\Omega))$ which are bounded
on $\overline{I}$.

 The space  $\lrm^p(I;\HH^k(\Omega))$, for $1\leq p\leq+\infty$,
and $k \in \mathbb{N}$, consists of $p$-integrable functions on $ I$
with values in $\HH^k(\Omega)$. For $1\leq p\leq+\infty$, $k\in
\mathbb{N}$ and  $0<\T\leq\infty$, the norm in
$\lrm^p((0,\T),\HH^k(\Omega))$  is denoted by
$\none{\,\cdot\,}_{p,k,\T}$.
$\lrm^2_{\textup{loc}}(\mathbb{R}_+;\HH^k(\Omega))$ is the set of
functions which are in $\lrm^2(I;\HH^k(\Omega))$ for all bounded
interval $I$ in $\mathbb{R}_+$.


%
%
 The letters $C$, $c_i$ or  $c_i^j$, $i,j=1,2,\cdots$, will denote
  constants taking different values, but not depending on $\ep$.
$\C_\Omegapt$ will be a constant, taking different values, and
depending only on $\Omega$. $(2.1)_n$ denotes the $n-$th equation of
System (2.1).

\subsection{The Main Result}
Recall the problem under study:
\begin{equation}\label{e8} \left\{
\begin{array}{r c l l}
\alpha\Bigl({\uu'}+(\uu\cdot\nabla)\uu\Bigl)+(1-\omega)A\uu+\nabla\sigma&=&
\FF(\uu,\sigma,\tau)+\textbf{\dv} \tau,\\
{\sigma'}+(\uu\cdot\nabla)\sigma+\sigma\dv\uu&=&-\ep^{\two}\alpha\,\dv{\uu},\\ \tau+\we\Bigl(\tau'+(\uu\cdot\nabla)\tau+\mathbf g(\nabla\uu,\tau)\Bigl)
&=&2\omega\D[\uu],\quad\text{ in }\Q_\T,\\
  \uu(0,\cdot)&=&\uu_\zero,&\text{in }\Omega,\\
\sigma(0,\cdot)&=&\sigma_\zero,&\text{in }\Omega,\\
\tau(0,\cdot)&=&\tau_\zero,&  \text{in }\Omega,\\
\uu&=&0,&\text{on } \Sigma_\T,\\
\end{array}
\right.
\end{equation}
where $\FF$ is defined by  (\ref{defF}).

\begin{thm}\label{thm2}\textup{(Existence   of a local solution)}
Assume $\Omega\subset \mathbb{R}^\three$ is a domain of class
$\C^\three$. Let $\mathfrak{m}_\one$ and $\mathfrak{M}_\one$ be two
real constants such that $0<\mathfrak{m}_\one
\leq\mathfrak{M}_\one$. Assume
\begin{eqnarray*}
\ff\in \lrm^2_{\textup{loc}}(\mathbb{R}_+;\HH^{1}(\Omega)),\,\,
with \, \ff'\in
\lrm^2_{\textup{loc}}(\mathbb{R}_+;\HH^{-1}(\Omega)),\, \uu_\zero
\in \HH^2(\Omega)\cap\HH^1_0(\Omega),\, \tau_\zero\in
\HH^2(\Omega),\\\\
\sigma_0\in \HHH^2(\Omega),\,\, with \,\,\int_\Omega
\sigma_0({\rm\bf x}) d{\rm\bf x}=0,\,\, {and} \,\,
0<\mathfrak{m}_\one\leq \alpha+\ep^2\sigma_0\leq \mathfrak{M}_\one
\,, \,\, in \,\, {\Omega}.\qquad
\end{eqnarray*}

Then there exists a time $\T_\one>0$ and a  solution
$(\uu,\sigma,\tau)$ of Problem (\ref{e7})-(\ref{ci1}) in
$Q_{\T_\one}=(0,\T_\one)\times\Omega$, satisfying
$$
\begin{array}{ll}
\uu \in \lrm^\two(0,\T_\one;\HH^\three(\Omega))\cap
\C([0,\T_\one];\HH^\two(\Omega)\cap\HH^\one_\zero(\Omega)),\,
\\
\uu' \in
\lrm^\two(0,\T_\one;\HH^\one_\zero(\Omega))\cap\C([0,\T_\one];\lbf^\two(\Omega)),
\\
(\tau,\sigma)\in\C([0,\T_\one];\HH^\two(\Omega)\times\HHH^\two(\Omega)), \quad (\tau',\sigma') \in
\C([0,\T_\one];\HH^\one(\Omega)\times\HHH^\one(\Omega)),
\end{array}
$$
with
\begin{equation*} \displaystyle{\int_\Omega\sigma(\cdot,\rm{x})d\rm{x}=0},   \quad { in }\,\,[0,\T_\one],\,\,
and\quad \frac{\mathfrak{m}_\one}{2}\leq \alpha\,+\ep^\two\sigma\leq
2\mathfrak{M}_\one \,, \quad { in }\,\, \overline{Q}_{\T_\one}.
\end{equation*}
\end{thm}
\begin{thm}\label{thm3}\textup{(Uniqueness of a local solution)}
There exist a unique solution of Problem (\ref{e7})-(\ref{ci1}), given in Theorem \ref{thm2}.
\end{thm}
To show that the local solution found in Theorem \ref{thm2} exists
for all times under certain regularity and smallness conditions on
the data, we also assume that the function $w\in\C^2(\mathbb{R})$
has the following properties: for all $h\in
\mathrm{L}^\two(0,\T;\mathrm{ H}^2(\Omega))$,
\begin{eqnarray}\label{condw}
\n{(w(h))'}\leq C\n{h'}, \qquad \n{w(h)} \leq C
\,\n{h},\nonumber\\
\n{w(h)}_k \leq C\,\n{h}_k,\quad k=1,2,
\end{eqnarray}
for some constant $C $ depending on $\Omega$ and $w$.
\begin{rem}
There are several examples of functions $p=p(\rho)$, for which $w$
satisfies the conditions above. Let us quote the case where the
pressure is given by  the linear state law
 $p(\rho)=\frac{1}{\ep^\two}(\rho-\alpha)$, as well as the case
of isothermal compressible perfect fluids, where $p(\rho)=(C_s)^\two
\rho$, and $C_s$  is the velocity of sound in the fluid.
\end{rem}
\section{{Existence and Uniqueness of Local Solutions}}\label{exloc-3}
We prove Theorem \ref{thm2} by using the classical method based on
the Schauder fixed-point theorem. To do that in our
case, we study three linear problems: the first one has the velocity
$\uu$ as unknown, and the next ones are two transport equations for
the density $\sigma$ and for the stress tensor $\tau$ respectively.
The parameter $\ep$ is fixed in the interval $(0,1]$.\\

Let $\ww$, $\pi$ and $\psi$  a given vector, function and the symmetric tensor of constraints
respectively. Let $\T$ a positive real number, $Q_{\T}=\Omega\times ]0,\T[$ and
$\Sigma_\T=\partial\Omega\times]0,\T[$. Consider the linear problem,
\begin{equation}\label{e9}
\left\{
\begin{array}{r c l l}
\alpha{\uu'}+(1-\omega)A\uu&=&
\mathfrak{F},\\
{\sigma'}+(\ww.\nabla)\sigma+\sigma\dv\ww&=&\mathcal{G},\\
\tau+\we\{\tau'+(\ww.\nabla)\tau+\mathbf g(\nabla\ww,\tau)\}
&=&2\omega\DD[\ww],&  \text{in }\Q_\T,\\
  \uu(0,x)&=&\uu_\zero(x),&\\
\sigma(0,x)&=&\sigma_\zero(x),&\\
\tau(0,x)&=&\tau_\zero(x),&  \text{in }\Omega,\\
\uu&=&0,&\text{ on } \Sigma_\T,
\end{array}
\right.
\end{equation}
with
\begin{eqnarray}
\mathfrak{F}&=&\FF(\ww,\pi,\psi)-\alpha(\ww.\nabla)\ww-\nabla\pi+\textbf{\dv} \psi,\label{defF1}\\
\mathcal{G}&=&-\ep^{-2}\dv\ww.\label{defG1}
\end{eqnarray}
and \begin{equation} \frac{\mathfrak{m}_\one}{2}\leq
\alpha+\ep^2\pi\leq 2\mathfrak{M}_\one\, , \quad \textup{ in }
\overline{Q}_{\T}. \label{condpi}
\end{equation}
\subsection{{Linear problem concerning the velocity} $\uu$}
Consider the linear problem concerning the velocity $\uu$,
\begin{equation}\label{e10}
\left\{\begin{array}{r c l l}
\alpha{\uu'}+(1-\omega)A\uu&=&
\mathfrak{F},&  \text{in }\Q_\T,\\
  \uu(0,x)&=&\uu_\zero(x),& \text{in }\Omega,\\
\uu&=&0,&\text{on } \Sigma_\T,
\end{array}
\right.
\end{equation}
where $A\uu=-(\Delta \uu+\nabla\dv \uu)$, $\mathfrak{F}$ and $\uu_\zero$
are given and $0<\T\leq +\infty$.\\

The first Lemma concerns the existence of a unique solution of
(\ref{e10}). By classical result of Agmon-Douglis-Nirenberg
\cite{ADN}, $A=-\Delta -\nabla \dv$ is a strongly elliptic operator,
and generates an analytic semigroup in $\lbf^\two(\Omega)$ with
domain $D(A)=\HH^\two(\Omega)\cap\HH^\one_\zero(\Omega)$ (we can see
for instance \cite{V}).
\begin{lem} \label{lema1} Let $ \Omega\subset
\mathbb{R}^\three$ of class $\C^\two$, $\mathfrak{F}\in
\lrm^\two(0,\T;\lbf^2(\Omega))$ and $\uu_\zero \in
\HH^\one_\zero(\Omega)$. Then there exists a unique solution of problem (\ref{e10})
\begin{eqnarray*}
\uu &\in& \lrm^\two(0,\T;\HH^\two(\Omega))\cap \C([0,\T];\HH^\one_\zero(\Omega)),\\
\uu'&\in& \lrm^\two(0,\T;\lbf^\two(\Omega)).
\end{eqnarray*}
Moreover, this solution satisfies the estimate
\begin{eqnarray}\label{estma1}
\frac{\alpha}{2}{\n{\uu'}^\two}_{\lebl{\textup{\tiny
2}}}+\frac{(1-\omega)^\two}{2}\n{A\uu}^\two_{\lebl{\two}}+{(1-\omega)}{\n{
D\uu }^\two}_{\lebl{\infty}}\hskip1cm\nonumber
\\+{(1-\omega)}{\n{\dv \uu }^\two}_{\lebl{\infty}}
 \leq 4(1-\omega)\n{D
\uu_\zero}^\two+\n{\mathfrak{F}}^\two_{\lebl{\two}}.\hskip1.5cm
\end{eqnarray}
\end{lem}
\begin{preuve}
By classical result of Agmon-Douglis-Nirenberg
\cite{ADN}, $A=-\Delta -\nabla \dv$ is a strongly elliptic operator,
and generates an analytic semigroup in $\lbf^\two(\Omega)$ with
domain $D(A)=\HH^\two(\Omega)\cap\HH^\one_\zero(\Omega)$ (we can see
for instance \cite{V}).

We start by showing the estimate (\ref{estma1}). Multiply $(\ref{e10})_\one$
in $\lbf^\two(\Omega)$ by
$\uu'+{\alpha(1-\omega)}A\uu$, then
\begin{eqnarray*}
\int_{\Omega} \abs{\uu'}^\two+2{(1-\omega)}\int_{\Omega}\uu'\cdot
A\uu+{(1-\omega)^\two}\int_{\Omega} \abs{A\uu}^\two \hskip3cm\\=
\int_{\Omega} \mathfrak{F}\cdot
\uu'+{(1-\omega)}\int_{\Omega}\mathfrak{F}\cdot A\uu.
\end{eqnarray*}
Integrate by parts the second term, we obtain
\begin{eqnarray*}
\n{\uu'}^\two+{(1-\omega)}\Dt{}\Big(\n{
D\uu}^\two+\n{\dv{\uu}}^\two\Big)+{(1-\omega)^\two}
\n{A\uu}^\two\leq
\n{\mathfrak{F}}\cdot\n{\uu'}\hskip3cm\\
+\frac{(1-\omega)}{4}\n{\mathfrak{F}}\cdot\n{A\uu}.
\end{eqnarray*}
On the other hand, using Young's inequality on the two terms
right, we get
\begin{eqnarray*}
\n{\mathfrak{F}}\cdot\n{\uu'}&\leq&
\frac{1}{2}\n{\mathfrak{F}}^\two+\frac{1}{2}\n{\uu'}^\two,\\
{(1-\omega)}\n{\mathfrak{F}}\cdot\n{A\uu}&\leq&
\frac{1}{2}\n{\mathfrak{F}}^\two+\frac{(1-\omega)^\two}{2}\n{A\uu}^\two.\\
\end{eqnarray*}
Integrate over  $[0,\T]$ and use the inequality
\begin{equation*}
\n{\dv \uu_\zero}\leq 3\n{D \uu_\zero},
\end{equation*}
then we get (\ref{estma1}).
\end{preuve}
The second Lemma give some stronger estimates.
\begin{lem}[\cite{V, Ta2}] \label{lema2} Under the conditions of Lemma \ref{lema1} and if $\partial \Omega
 \in \C^\three$, $\mathfrak{F}\in
\lrm^\two(0,\T;\HH^\one(\Omega))$, $\mathfrak{F}'\in
\lrm^\two(0,\T;\HH^{-1}(\Omega))$ and $\uu_\zero \in
\HH^\one(\Omega)\cap\HH^\one_\zero(\Omega)$. Then the solution $\uu$
of problem (\ref{e10}) given by Lemma \ref{lema1} is such that
\begin{eqnarray*}
\uu &\in& \lrm^\two(0,\T;\HH^\three(\Omega))\cap
\C([0,\T];\HH^\two(\Omega)\cap\HH^\one_\zero(\Omega)),\\
 \uu' & \in&
\lrm^\two(0,\T;\HH^\one_\zero(\Omega))\cap
\C([0,\T];\lbf^\two(\Omega)).
\end{eqnarray*}
and there exists a constant $C_\one$, depend only in $\Omega$, $\T$,
$\alpha$ and $\omega$, such that one has the estimate
\begin{eqnarray}\label{estma2}
\n{\uu}^\two_{\lebhhh{\two}}+\n{\uu}^\two_{\lebhintersection{\infty}}+
\n{\uu'}^\two_{\lebhzero{\two}}+\n{\uu'}^\two_{\lebl{\infty}}\hskip1cm\nonumber\\
\nonumber\\\leq C_\one\{\n{A\uu_\zero}^\two +\n{\mathfrak{F}(0)}^\two+
\n{\mathfrak{F}}^\two_{\lebh{\two}}+
\n{ \mathfrak{F}'}^\two_{\lebhmoins{\two}}\}.\nonumber\\
\end{eqnarray}
\end{lem}
\begin{preuve}
Derive in terms of $t$ the equation $(\ref{e10})_1$, then we obtain
\begin{equation*}
{\uu''}+\alpha(1-\omega)A\uu'= \mathfrak{F}', \hskip2cm\text{ in
}\Q_\T,
\end{equation*}
and $\uu'_{|\partial \Omega}(t)=0$ for all $t\in [0,\T]$. Let
$\vv=\uu'$, then $\vv$ verify the system
\begin{equation}\label{e10bis}
\left\{
\begin{array}{r c l l}
{\vv'}+\alpha(1-\omega)A\vv&=&
\mathfrak{F}',&  \text{ in }\Q_\T,\\
  \vv(0)&=&\vv_\zero=\mathfrak{F}(0)-\alpha(1-\omega)A\uu_\zero,& \text{ in }\Omega,\\
\vv&=&0,&\text{ on } \Sigma_\T.
\end{array}
\right.
\end{equation}
Multiply by $\vv$ the equation $(\ref{e11})_1$ and integrate on
$\Omega$. It comes
\begin{eqnarray*}
\int_{\Omega} \vv'\cdot\vv+{\alpha(1-\omega)}\int_{\Omega}
A\vv\cdot\vv= <\mathfrak{F}',\vv>_{\HH^{-1},\HH^1_0}.
\end{eqnarray*}
After calculation, we obtain
\begin{eqnarray*}
\frac{1}{2}\Dt{}\n{\vv}^\two+\frac{3\alpha(1-\omega)}{4}\n{
D\vv}^\two+\alpha(1-\omega)\n{\dv{\vv}}^\two\leq
\frac{1}{\alpha(1-\omega)}\n{ \mathfrak{F}'}^\two_{\HH^{-1}}.
\end{eqnarray*}
Integrate on $[0,\T]$ and replace  $\vv$ and $\vv_0$ by their
values
\begin{eqnarray}\label{estm2-bis1}
\frac{1}{2}\n{\uu'}^\two_{\lebl{\infty}}+\frac{3\alpha(1-\omega)}{4}\n{
D\uu'}^\two_{\lebl{\two}}+\alpha(1-\omega)
\n{\dv{\uu'}}^\two_{\lebl{\two}}\hskip1cm\nonumber\\
\leq
\frac{1}{\alpha(1-\omega)}\n{ \mathfrak{F}'}^\two_{\lrm^\two(0,\T,\HH^{-1}(\Omega))}+\frac{1}{2}\Big(\n{\mathfrak{F}(0)}^\two+\alpha(1-\omega)
\n{A\uu_\zero}^\two\Big).\nonumber\\
\end{eqnarray}
 Finally, inequality  (\ref{estma2}) follows from inequality  (\ref{estma1}) and
 (\ref{estm2-bis1}).
\end{preuve}
\subsection{Resolution of the Transport Problems}
We consider the following two linear transport problems,
\begin{equation}\label{e11}
\left\{
\begin{array}{r c l l}
{\sigma'}+(\ww\cdot\nabla)\sigma+\sigma\,\dv\ww&=&-\ep^\mtwo \alpha \dv \ww,&  \text{ in }\Q_\T,\\
\sigma(0,\cdot)&=&\sigma_\zero,&\text{ in }\Omega,
\end{array}
\right.
\end{equation}
and
 \begin{equation}\label{e12}
\left\{
\begin{array}{r c l l}
\tau+\we\Bigl(\tau'+(\ww\cdot\nabla)\tau+\mathbf
g(\nabla\ww,\tau)\Bigl)
&=&2\omega\D[\ww],&  \text{ in }\Q_\T,\\
  \tau(0,\cdot)&=&\tau_\zero,&  \text{ in }\Omega,
\end{array}
\right.
\end{equation}
where  $\sigma_\zero$  and $\tau_0$ are, respectively, some given
function and symmetric tensor defined in $\Omega$. The existence of
solutions to this problems follows from the classical method of
characteristics. (see for example \cite{GS,Ta2,V}).
 The lemmas below
give some estimates of the solutions of these problems.
\begin{lem}[\cite{V}] \label{lem3} Let $\Gamma
\in \C^\one$, $\ww \in \lrm^\one(0,\T,\HH^\three(\Omega))$,
$\ww.\nn=0$ on $\Sigma_\T$, and $\sigma_\zero\in \HH^\two(\Omega)$,
with $\displaystyle{\int_\Omega\sigma_\zero\,d\textbf{x}}=0$. Then
there exists a unique solution
$\sigma\in\C([0,\T];\HHH^\two(\Omega))$ of (\ref{e11}) such that
\begin{eqnarray*}
\int_\Omega\sigma(\cdot,{\rm\bf x})d{\bf\rm x}=0 {\quad in \quad}
[0,\T],
\end{eqnarray*}
%
%
and satisfying the following estimate
\begin{equation*}\label{estm3}
\n{\sigma}_{\lebhh{\infty}}\leq (\n{\sigma_\zero}+{\alpha}{\ep}^\mtwo)\exp
\big(C_\Omegapt\n{\ww}_{\lebhhh{\one}}\big),
\end{equation*}
for some positive constant $C_\Omegapt$ depending on $\Omega$.

If, in addition, $\ww \in \C([0,\T];\HH^\two(\Omega))$, then $\sigma'
\in \C([0,\T];\HHH^1(\Omega))$ satisfies
\begin{equation*}\label{estm4}
\n{\sigma'}_{\lebh{\infty}}\leq
C_\Omegapt\n{\ww}_{\lebhh{\infty}}(\n{\sigma_\zero}+{\alpha}{\ep}^\mtwo)\exp
\Big(C_\Omegapt\n{\ww}_{\lebhhh{\one}}\Big).
\end{equation*}
\end{lem}
\begin{lem}[\cite{GS}] \label{lem4} Let $\Omega
\subset \mathbb{R}^\three$ be a  domain of class  $\C^\three$, $\ww
\in \lrm^\one(0,\T;\HH^\three(\Omega)\cap \HH^\one_\zero(\Omega))$
and $\tau_\zero\in \HH^\two(\Omega)$. Then there exists a unique
solution $\tau \in \C([0,\T];\HH^\two(\Omega))$ of (\ref{e12}), such
that
\begin{equation*}\label{estm5}
\n{\tau}_{\lebhh{\infty}}\leq
\Big(\n{\tau_\zero}^\two+\frac{2\omega}{C_\Omegapt\we}\Big)\exp\Big(C_\Omegapt\n{\ww}_{\lebhhh{\one}}\Big),
\end{equation*}
for some positive constant $C_\Omegapt$ depending on $\Omega$.

If, in addition, $\ww \in
\C([0,\T];\HH^\two(\Omega)\cap\HH^\one_\zero(\Omega))$,  then $\tau'
\in \C([0,\T];\HH^\one(\Omega))$ satisfies
\begin{equation*}\label{estm6}
\n{\tau'}_{\lebh{\infty}}\leq
C_\zero\Big(\n{\ww}_{\lebhh{\infty}}+\frac{1}{C_\Omegapt\we}\Big)\Big(\n{\tau_\zero}
+\frac{2\omega}{C_\Omegapt\we}\Big)\exp\Big(C_\Omegapt\n{\ww}_{\lebhhh{\one}}\Big).
\end{equation*}
\end{lem}
\subsection{{Proof of Theorem }\ref{thm2}}
We are now in a position to prove the local existence of a solution to problem (\ref{e8}). We apply  the Theorem of fixed-point of Schauder.\\

 Take $\T>0$, $\mathfrak{B}_\one$, $\mathfrak{B}_\two>0$, and define
\begin{eqnarray*}
\mathfrak{R}_\T&=&\{(\ww,\pi,\psi), \\&&\ww\in \conthhintersection
\cap \lebhhhbf{\infty},\ww'\in \contllbf\cap \lebhzerobf{\two}\\
&&\pi\in\lebhh{\infty},\pi'\in \lebh{\infty},\\
&&\psi\in\lebhhbf{\infty},\psi'\in \lebhbf{\infty},
\\
&&\ww(0)=\uu_\zero,\,\pi(0)=\sigma_0,\psi(0)=\tau_\zero \textup{
in } \Omega,
\ww=0 \textup{ in }\Sigma_\T,\\
&&\n{\ww}^\two_{\lebhintersection{\infty}}+\n{\ww}^\two_{\lebhhh{\two}}+\n{\ww'}^\two_{\lebl{\infty}}+\n{\ww'}^\two_{\lebhzero{\two}}\leq\mathfrak{B}_\one,
\\&&\n{\pi}_{\lebhh{\infty}}+\n{\psi}_{\lebhh{\infty}}\leq
\mathfrak{B}_\one,
\\&&\n{\pi'}_{\lebh{\infty}}+\n{\psi'}_{\lebh{\infty}}\leq
\mathfrak{B}_\two,
\\&&\frac{\mathfrak{m}_\one}{2}\leq
\alpha+\ep^\two\pi(t,x)\leq 2\mathfrak{M}_\one\, , \quad \textup{
in } \overline{Q}_{\T}\}.
\end{eqnarray*}
Choose $\mathfrak{B}_\one$ such that
\begin{equation}\label{choix-B1-1}
\mathfrak{B}_\one>\max\{C_\four\n{A\uu_\zero}^\two,\n{\sigma_0}_\two,\n{\tau_\zero}_\two\},
\end{equation}
then $(\uu_0,\sigma_0,\tau_\zero)\in \mathfrak{R}_\T$. In fact,
$\ww$ is a solution of problem
\begin{equation}\label{e10-tri}
\left\{
\begin{array}{r c l l}
\ww(\cdot) \in \HH^\one(\Omega),\\
{\ww'}+(1-\omega)A\ww&=&
0,&   \text{ p.p. in } \mathbb{R}_+,\\
  \ww(0)&=&\uu_\zero,& \text{in }\Omega,\\
\ww&=&0,&\text{on } \Sigma_\T.
\end{array}
\right.
\end{equation}
Using estimate (\ref{estma2}), there exists a constant $C_4$
such that
\begin{eqnarray*}\label{estm2}
\n{\ww'}^\two_{\lebhhhbf{\two}}+\n{\ww'}^\two_{\lebhintersectionbf{\infty}}+\n{\ww}^\two_{\lebhzerobf{\two}}+\n{\ww}^\two_{\leblbf{\infty}}\leq
 C_4\n{A\uu_\zero}^\two.
\end{eqnarray*}
Thus the choose of $\mathfrak{B}_\one$ in (\ref{choix-B1-1}) is
enough for prove that $\mathfrak{R}_\T$ is non empty for each
$\T>0$.\\ \\Define now the application mapping $\mathfrak{K}$ in
this way
$$\begin{array}{rccl}
  \mathfrak{K}: & \mathfrak{R}_\T & \longrightarrow & \mathfrak{X}_\T=\conthzerobf\times \conth \times \conthbf\\
    &(\ww,\pi,\psi) & \longrightarrow & (\uu,\sigma,\tau)
  \end{array}
$$
where $\uu$, $\sigma$ and $\tau$ are solution of
 (\ref{e10}), (\ref{e11}) and (\ref{e12}), respectively, with
 \begin{eqnarray*}
\mathfrak{F}&=&\alpha\ff+(1-\omega)\frac{\ep^\two\pi}{\alpha+\ep^\two\pi}A\ww+\frac{\ep^\two}
{\alpha+\ep^\two\pi}(\pi-w(\pi))\nabla\pi-\alpha(\ww.\nabla)\ww-\nabla\pi+ \textbf{\dv} \psi,\label{defF1}\\
\mathcal{G}&=&-\ep^{-2}\alpha\,\dv\ww.\label{defG1}
\end{eqnarray*}
If we take
\begin{eqnarray}\label{choix-B1-2}
\mathfrak{B}_\one&>&\max\Big\{C_\four\n{A\uu_\zero}^\two,e^{\sqrt{2}}
\left(\n{\sigma_\zero}_\two+\n{\tau_\zero}_\two+1+\frac{2\omega}{C_\three\we}\right),C_\two(2C_\five+1)\n{A\uu_\zero}^\two\nonumber\\
&&+C_\five\n{A\uu_\zero}^\four+3\Big(2(1+\n{w}_\mathcal{C}^\two)\n{\sigma_\zero}_\one^\two+\n{\tau_\zero}_\one^\two\Big)\nonumber\\
&&+3\n{\ff(0)}^\two+3\n{\ff}^\two_{\lebh{\two}}+3\n{\ff'}^\two_{\lebhmoins{\two}}
\Big \},
\end{eqnarray}
and
\begin{eqnarray}\label{choix-B2-1}
\mathfrak{B}_\two&>&e^{\sqrt{2}}\left\{C_\six\left(\n{\sigma_\zero}_\two+\n{\tau_\zero}_\two+1+\frac{2\omega}{C_\three\we}\right)+\frac{1}{\we}
\left(\n{\tau_\zero}_\two+ \frac{2\omega}{C_\three\we}\right)
 \right\},
\end{eqnarray}
and for all  $\T$ small enough such that
\begin{equation}
\T\leq
\T^*=\min\left(\frac{\mathfrak{B}_\one}{2C_\two(4C_\four(1+\n{w}_\mathcal{C}^\two)\mathfrak{B}_\one^\two+3\mathfrak{B}_\two^\two},
\frac{2}{C_\six^\two\mathfrak{B}_\one}\right),
\end{equation}
we have $\mathfrak{K}(\mathfrak{R}_{\T})\subset
\mathfrak{R}_{\T}$.\\

We now use Schauder fixed point theorem. The mapping $ \mathfrak{K}$
is defined from  convex, bounded and no empty set
$\mathfrak{R}_{\T}$ into $\mathfrak{X}_\T$. To finish, we need to  show the continuity of $ \mathfrak{K}$ in $\mathfrak{X}_\T$.
\begin{lem}
To  show the continuity of $ \mathfrak{K}$ in $\mathfrak{X}_\T$, it is enough
to show the continuity of $ \mathfrak{K}$ in
$$\mathfrak{Y}_\T=C([0,\T]; \mathbf{L}^\two(\Omega))\times C([0,\T];
\mathrm{L}^\two(\Omega))\times C([0,\T];
\mathbf{L}^\two(\Omega)).$$
\end{lem}
\begin{proof}
Let $\Big((\ww_n,\pi_n,\psi_n)\Big)_n$ be a sequence of   $\mathfrak{R}_{\T}$ and tends to $(\ww,\pi,\psi)$, such that:
$$(\uu_n,\sigma_n,\tau_n)=\mathfrak{K}(\ww_n,\pi_n,\psi_n) \textup{ and } (\uu,\sigma,\tau)=\mathfrak{K}(\ww,\pi,\psi).$$
Suppose that $\mathfrak{K}$ is continuous in $\mathfrak{Y}_\T$, then the sequence $\Big(\mathfrak{K}(\ww_n,\pi_n,\psi_n)_n\Big)$ tends to $\mathfrak{K}(\ww,\pi,\psi)$ in $\mathfrak{Y}_\T$, \textit{i.e.}
\begin{equation}\label{lim1} \lim_{n \longrightarrow \infty}\n{(\uu_n,\sigma_n,\tau_n)-(\uu,\sigma,\tau)}_{\mathfrak{Y}_\T}=0.
\end{equation}
$\mathfrak{R}_{\T}$ is a compact set in  $\mathfrak{X}_\T$ (see for instance \cite{S}). Using (\ref{lim1}), we can extract of $\Big((\uu_n,\sigma_n,\tau_n)\Big)_n$ a subsequence converges in $\mathfrak{X}_\T$ to  the unique accumulation point $(\uu,\sigma,\tau)$. Then the sequence $\Big((\uu_n,\sigma_n,\tau_n)\Big)_n=\Big(\mathfrak{K}(\ww_n,\pi_n,\psi_n)\Big)_n$ converges to $(\uu,\sigma,\tau)=\mathfrak{K}(\ww,\pi,\psi)$ in $\mathfrak{X}_\T$. This proved the  continuity of $\mathfrak{K}$  in
$\mathfrak{X}_\T$.
\end{proof}
\begin{lem}
$ \mathfrak{K}$ is continuous in $\mathfrak{Y}_\T$.
\end{lem}
\begin{preuve}
Let $\Big((\ww_n,\pi_n,\psi_n)\Big)_n$ be a sequence of   $\mathfrak{R}_{\T}$ and tends to $(\ww,\pi,\psi)$, such that:
$$(\uu_n,\sigma_n,\tau_n)=\mathfrak{K}(\ww_n,\pi_n,\psi_n) \textup{ and } (\uu,\sigma,\tau)=\mathfrak{K}(\ww,\pi,\psi).$$
Consider   two systems. The first is :
\begin{equation}\label{e9-1}
\left\{
\begin{array}{r c l l}
\alpha{\uu_n'}+(1-\omega)A\uu_n&=&
 \mathfrak{F}_n,\\
{\sigma_n'}+(\ww_n.\nabla)\sigma_n+\sigma_n\dv\ww_n&=&\mathcal{G}_n,\\
\tau_n+\we\{\tau_n'+(\ww_n.\nabla)\tau_n+\mathbf g(\nabla\ww_n,\tau_n)\}
&=&2\omega\DD[\ww_n],&  \text{in }\Q_\T,\\
  \uu_n(0,x)&=&\uu_\zero(x),&\\
\sigma_n(0,x)&=&\sigma_\zero(x),&\\
\tau_n(0,x)&=&\tau_\zero(x),&  \text{in }\Omega,\\
\uu_n&=&0,&\text{on } \Sigma_\T,
\end{array}
\right.
\end{equation}
with
\begin{eqnarray*}
\mathfrak{F}_n&=&\FF(\ww_n,\pi_n,\psi_n)-\alpha(\ww_n.\nabla)\ww_n-\nabla\pi_n+\textbf{\dv} \psi_n,\label{defF1-2}\\
\mathcal{G}_n&=&-\ep^{-2}\dv\ww_n.\label{defG1-2}
\end{eqnarray*}
And, the second is:
\begin{equation}\label{e9-2}
\left\{
\begin{array}{r c l l}
\alpha{\uu'}+(1-\omega)A\uu&=&
\mathfrak{F},\\
{\sigma'}+(\ww.\nabla)\sigma+\sigma\dv\ww&=&\mathcal{G},\\
\tau+\we\{\tau'+(\ww.\nabla)\tau+\mathbf g(\nabla\ww,\tau)\}
&=&2\omega\DD[\ww],&  \text{in }\Q_\T,\\
  \uu(0,x)&=&\uu_\zero(x),&\\
\sigma(0,x)&=&\sigma_\zero(x),&\\
\tau(0,x)&=&\tau_\zero(x),&  \text{in }\Omega,\\
\uu&=&0,&\text{on } \Sigma_\T,
\end{array}
\right.
\end{equation}
with
\begin{eqnarray*}
\mathfrak{F}&=&\FF(\ww,\pi,\psi)-\alpha(\ww.\nabla)\ww-\nabla\pi+\textbf{\dv} \psi,\label{defF1}\\
\mathcal{G}&=&-\ep^{-2}\dv\ww.\label{defG1}
\end{eqnarray*}
Let $\vv_n=\uu_n-\uu$, $q_n=\sigma_n-\sigma$ and $\SSS_n=\tau_n-\tau$. Using (\ref{e9-1}) and (\ref{e9-2}), we obtain, in $\Q_\T$:
\begin{equation}\label{e9-3}
\left\{
\begin{array}{r c l l}
\alpha{\vv_n'}+(1-\omega)A\vv_n&=&
\mathfrak{F}_1,\\
{q_n'}+(\ww_n.\nabla)q_n+
q_n\dv\ww_n&=&\mathcal{G}_1-\big((\ww_n-\ww).\nabla\big)\sigma-\sigma\dv(\ww_n-\ww),\\
\SSS_n+\we\{\SSS_n'+(\ww_n.\nabla)\SSS_n+\mathbf g(\nabla\ww_n,\SSS_n)\}
&=&\mathcal{H}_1-\we\Big\{\big((\ww_n-\ww).\nabla\big)\tau+\mathbf g\big(\nabla(\ww_n-\ww),\tau\big)\Big\},
\end{array}
\right.
\end{equation}
with the boundary conditions:
\begin{equation}\label{e9-4}
\left\{
\begin{array}{r c l l}
  \vv_n(0,x)&=&0,&\\
q_n(0,x)&=&0,&\\
\SSS_n(0,x)&=&0,&  \text{in }\Omega,\\
\vv_n&=&0,&\text{on } \Sigma_\T,
\end{array}
\right.
\end{equation}
such that:
\begin{eqnarray*}
\mathfrak{F}_1&=&\FF(\ww_n,\pi_n,\psi_n)-\FF(\ww,\pi,\psi)-\alpha\Big[(\ww_n.\nabla)\ww_n-(\ww.\nabla)\ww\Big]-\nabla(\pi_n-\pi)+\textbf{\dv}( \psi_n-\psi),\label{defF1}\\
\mathcal{G}_1&=&-\ep^{-2}\dv(\ww_n-\ww),\label{defG1}\\
\mathcal{H}_1&=&2\omega\DD[\ww_n-\ww].
\end{eqnarray*}
First, multiply  the  equation  $(\ref{e9-3})_\one$  by $\vv_n$, 
and integrate over $\Omega$. We get:
\begin{eqnarray}\label{e10-1}
\alpha \Dt{}\n{\vv_n}^\two+(1-\omega)\Big(\n{\nabla\vv_n}^\two+\n{\dv \vv_n}^\two\Big)&\leq& \n{\pi_n-\pi}_\one^\two+\n{\psi_n-\psi}_\one^\two+4\n{\vv_n}^\two\nonumber\\
&&+\alpha^\two\n{(\ww_n.\nabla)\ww_n-(\ww.\nabla)\ww}^\two \nonumber\\
&&+\n{\FF(\ww_n,\pi_n,\psi_n)-\FF(\ww,\pi,\psi)}^\two.
\end{eqnarray}
We now estimate the term $\n{\FF(\ww_n,\pi_n,\psi_n)-\FF(\ww,\pi,\psi)}^\two$ on the right hand of   (\ref{e10-1}). Using the two inequalities :
\begin{eqnarray*}\label{e10-2}
\frac{\mathfrak{m}_\one}{2}\leq
\alpha+\ep^\two\pi(t,x)\leq 2\mathfrak{M}_\one \textup{ and }
\frac{\mathfrak{m}_\one}{2}\leq
\alpha+\ep^\two\pi_n(t,x)\leq 2\mathfrak{M}_\one,
\end{eqnarray*}
we obtain:
\begin{eqnarray*}\label{e10-3}
\n{\FF(\ww_n,\pi_n,\psi_n)-\FF(\ww,\pi,\psi)}^\two\leq C_\seven \mathfrak{m}_\one \ep^\four\Big[(1-\omega)^\two\n{\pi_n A\ww_n-\pi A\ww}^\two+\n{\pi_n\nabla \pi_n-\pi\nabla \pi}^\two\\+\n{w(\pi_n)\nabla \pi_n-w(\pi )\nabla \pi }^\two+\n{\pi_n \textbf{\dv}\SSS_n-\pi  \textbf{\dv}\SSS}^\two\Big].
\end{eqnarray*}
Then, (\ref{e10-1}) satisfies:
\begin{eqnarray}\label{e10-4}
\alpha \Dt{}\n{\vv_n}^\two+(1-\omega)\Big(\n{\nabla\vv_n}^\two+\n{\dv \vv_n}^\two\Big)&\leq& C_\eight\ell_n+ +4\n{\vv_n}^\two,
\end{eqnarray}
with
\begin{eqnarray*}\label{e10-5}
\ell_n&=&\n{\pi_n-\pi}_\one^\two+\n{\psi_n-\psi}_\one^\two +\alpha\n{(\ww_n.\nabla)\ww_n-(\ww.\nabla)\ww}^\two+(1-\omega)^\two\n{\pi_n A\ww_n-\pi A\ww}^\two\nonumber\\
&&+\n{\pi_n\nabla \pi_n-\pi\nabla \pi}^\two+\n{w(\pi_n)\nabla \pi_n-w(\pi )\nabla \pi }^\two
 +\n{\pi_n \textbf{\dv}\SSS_n-\pi  \textbf{\dv}\SSS}^\two.
\end{eqnarray*}
Second, multiply  the  equation  $(\ref{e9-3})_\two$  by $\ep^\two q_n$ and integrate over $\Omega$. This yields:
\begin{eqnarray}\label{e11-1}
\ep^\two\Dt{}\n{q_n}^\two&\leq&  (1+C_\nine\n{\sigma}_\two)\n{\ww_n-\ww}_\one^\two+(1+C_\ten  \n{\ww_n}_\three)\n{q_n}^\two\nonumber\\
&\leq&  (1+C_\nine\n{\sigma}_\two)\n{\ww_n-\ww}_\one^\two+\textit{j}_n\n{q_n}^\two,
\end{eqnarray}
with $j_n=1+C_\ten  \n{\ww_n}_\three$.\\
Finally, multiply  the  equation  $(\ref{e9-3})_\three$  by $\SSS_n/2\omega$, we obtain
\begin{eqnarray}\label{e12-1}
\frac{\we}{2\omega}\Dt{}\n{\SSS_n}^\two&\leq &\left(1+ \frac{\we C_\eleven}{2\omega}\n{\tau}_\two\right)\n{\ww_n-\ww}_\one^\two+ \left(1+\frac{1}{2\omega}+ \frac{\we  C_\twelve}{2\omega} \n{\ww_n}_\three \right)\n{\SSS_n}^\two\nonumber\\
&\leq &\left(1+ \frac{\we C_\eleven}{2\omega}\n{\tau}_\two\right)\n{\ww_n-\ww}_\one^\two+ \textit{k}_n\n{\SSS_n}^\two,
\end{eqnarray}
with $\textit{k}_n=1+\frac{1}{2\omega}+ \frac{\we  C_\twelve}{2\omega} \n{\ww_n}_\three$.

The functions $\textit{j}_n$  and $\textit{k}_n$, are positive and, because of the class
of solutions we consider,  $\textit{j}_n$  and $\textit{k}_n$  belong to $L^\one(0, T)$. Therefore, by using (\ref{e10-4}), (\ref{e11-1}), and (\ref{e12-1}),  we
deduce from Gronwall's lemma that:
\begin{eqnarray}
\n{\vv_n}^\two&\leq& \frac{C_\eight}{\alpha}\int_0^t \exp\left(\frac{-4s}{\alpha}\right)\ell_n(s) ds,\label{e11-2}\\
\n{q_n}^\two&\leq&\frac{1}{\ep^\two}\left(\1+C_\nine \mathfrak{B}_\one\right)\int_s^t \exp\left(\int_0^s \textit{j}_n(r) dr\right)\n{\ww_n(s)-\ww(s)}_\one^\two ds,\label{e11-3} \\
\n{\SSS_n}^\two&\leq&\left(\frac{2\omega}{\we}+C_\eleven \mathfrak{B}_\one\right)\int_s^t \exp\left(\int_0^s \textit{k}_n(r) dr\right)\n{\ww_n(s)-\ww(s)}_\one^\two ds.\label{e11-4}
\end{eqnarray}
The  sequence $\Big((\ww_n,\pi_n,\psi_n)\Big)_n$  of  $\mathfrak{R}_{\T}$   tends to $(\ww,\pi,\psi)$, and using (\ref{e11-2}), (\ref{e11-3}) and (\ref{e11-4}), we obtain  $\vv_n$, $q_n$ and $\SSS_n$ tend to zero in $\mathfrak{Y}_\T$. This meaning that the  sequence $\Big((\uu_n,\sigma_n,\tau_n)\Big)_n=\Big(\mathfrak{K}(\ww_n,\pi_n,\psi_n)\Big)_n$ tends to $ (\uu,\sigma,\tau)=\mathfrak{K}(\ww,\pi,\psi)$ and $ \mathfrak{K}$ is continuous in $\mathfrak{Y}_\T$.
\end{preuve}
\subsection{{Proof of Theorem }\ref{thm3}}
 We take, as usual, the difference of two solutions $(\uu_\one,\sigma_\one,\tau_\one)$ and $(\uu_\two,\sigma_\two,\tau_\two)$ belonging
to the class specified in the theorem \ref{thm3}. The vector function $\uu=\uu_\one-\uu_\two$, the scalar function $\sigma=\sigma_\one-\sigma_\two$ and the tensor function $\tau=\tau_\one-\tau_\two$ satisfy the following system:\begin{equation}\label{e13-1}
\left\{
\begin{array}{r c l l}
\alpha\Big[{\uu'}+
 (\uu .\nabla)\uu_\one+(\uu_\two.\nabla)\uu \Big]+(1-\omega)A\uu&=&
\mathfrak{F}_2 -\nabla\sigma+\textbf{\dv}\tau,\\
{\sigma'}+(\uu_\one.\nabla)\sigma+ (\uu.\nabla)\sigma_\two+
\sigma\dv\uu_\one+\sigma_\two\dv\uu&=&-\ep^{-2}\dv\uu,\\
\tau+\we\{\tau'+(\uu_\one.\nabla)\tau+\big(\uu.\nabla\big)\tau_\two+\mathbf g(\nabla\uu_\one,\tau )+\mathbf g\big(\nabla\uu,\tau_\two\big)\}
&=&2\omega\DD[\uu],
\end{array}
\right.
\end{equation}
with the boundary conditions:
\begin{equation}\label{e13-2}
\left\{
\begin{array}{r c l l}
  \uu(0,x)&=&0,&\\
\sigma(0,x)&=&0,&\\
\tau(0,x)&=&0,&  \text{in }\Omega,\\
\uu&=&0,&\text{on } \Sigma_\T,
\end{array}
\right.
\end{equation}
such that:
\begin{eqnarray*}
\mathfrak{F}_2&=&\FF(\uu_\one,\sigma_\one,\tau_\one)-\FF(\uu_\two,\sigma_\two,\tau_\two).\label{defF1}\\
\end{eqnarray*}
Multiply $(\ref{e13-1})_\one$, $(\ref{e13-1})_\two$ and $(\ref{e13-1})_\three$ by
$\uu$, $\ep^\two \sigma/\alpha$, and
 ${\tau}/({2 \omega})$, respectively,  and integrate over $\Omega$. Summing the three obtained  equations, one
 obtains
\begin{eqnarray}\label{estm13-3}
&&\frac{1}{2}\Dt{}\left(\alpha{\n{\uu}^\two}+\frac{\ep^\two}{\alpha\,}{\n{\sigma}^\two}+\frac{\we}{2
\omega}{\n{\tau}^\two}\right)+(1-\omega)\Bigl(\n{\nabla
\uu}^\two+\n{\dv \uu}^\two\Bigl)+\frac{1}{2
\omega}\n{\tau}^\two\hskip2cm\nonumber\\
&&\quad \leq \alpha \C_\twelve\Big[ \n{\uu_\one}\n{\uu}^\two+\n{\uu_\two}\n{\uu}^\two\Big]+
 \frac{\ep^\two}{\alpha} \C_\twelve\Big[\n{\sigma}\n{\uu_\one}_\three\n{\nabla\uu}+
 \n{\sigma_\two}_\two\n{\nabla\uu}\n{\uu}\\
&&\qquad+\n{\sigma}\n{\sigma_\one}_\two\n{\dv \uu}+\n{\sigma_\two}_\two\n{\sigma}\n{\dv\uu}+\n{\tau_\one}\n{\sigma}\n{\nabla\uu}+\n{\sigma_\two}\n{\tau}\n{\nabla\uu}\Big]\nonumber\\
 &&\qquad +\frac{\ep^\two }{ \alpha\,} \C_\twelve \Big[\n{\uu_\one}_\three\n{\sigma}^\two+\n{\nabla \uu}\n{\sigma_\two}_\two\n{\sigma}\Big] +\frac{\we }{2 \omega}\C_\twelve
 \Big[\n{\uu_\one}_\three\n{\tau}^\two+\n{\nabla \uu}\n{\tau_\two}_\two\n{\tau}\Big]. \nonumber
\end{eqnarray}
For $\delta> 0$, (\ref{estm13-3}) can be written as:
\begin{eqnarray}\label{estm13-4}
&&\frac{1}{2}\Dt{}\left(\alpha{\n{\uu}^\two}+\frac{\ep^\two}{\alpha\,}{\n{\sigma}^\two}+\frac{\we}{2
\omega}{\n{\tau}^\two}\right)+(1-\omega)\Bigl(\n{\nabla
\uu}^\two+\n{\dv \uu}^\two\Bigl)+\frac{1}{2
\omega}\n{\tau}^\two\hskip2cm\nonumber\\
&&\quad \leq \left[\C_\twelve\bigl( \n{\uu_\one} +\n{\uu_\two}\bigl)+\frac{(\C_\twelve)^\two}{2\delta}\n{\sigma_\two}^\two_\two\right]\alpha \n{\uu}^\two\nonumber\\
&&\qquad+\frac{\delta}{2}\left[\left(\frac{5\ep^\two}{\alpha}+\frac{\we}{2\omega}\right)\n{\nabla\uu}^\two+\frac{2\ep^\two}{\alpha}\n{\dv\uu}^\two\right]
\nonumber\\
&&\qquad +\left[\frac{(\C_\twelve)^\two}{2\delta}\Big(\n{\uu_\one}^\three_\two +\n{\sigma_\one}^\two_\two+2\n{\sigma_\two}^\two_\two+\n{\tau_\one}^\two_\two\Big)+\C_\twelve\n{\uu_\one}_\three\right]\frac{\ep^\two}{\alpha}
\n{\sigma}^\two\nonumber\\
&&\qquad +\left[\frac{(\C_\twelve)^\two}{2\delta}\n{\sigma_\two}^\two_\two+\C_\twelve\n{\uu_\one}_\three\right]\frac{\we}{2\omega}
\n{\tau}^\two.
\end{eqnarray}
From (\ref{estm13-4}), we then deduce that solutions $(\uu,\sigma,\tau)$ of (\ref{e13-1}) satisfy the following energy inequality:
\begin{eqnarray}\label{estm13-5}
&&\frac{1}{2}\Dt{}\left(\alpha{\n{\uu}^\two}+\frac{\ep^\two}{\alpha\,}{\n{\sigma}^\two}+\frac{\we}{2
\omega}{\n{\tau}^\two}\right)+(1-\omega)\Bigl(1- \frac{\delta(10\ep^\two\omega + \alpha{\we})}{4\alpha \omega(1-\omega)}\Bigl)\n{\nabla
\uu}^\two\nonumber\\
&&\qquad +(1-\omega)\Bigl(1- \frac{\delta\ep^\two}{\alpha(1-\omega)}\Bigl) \n{\dv \uu}^\two +\frac{1}{2
\omega}\n{\tau}^\two  \leq  {\mathcal{X} }_\delta\left[ \alpha \n{\uu}^\two+\frac{\ep^\two}{\alpha\,}{\n{\sigma}^\two}+\frac{\we}{2
\omega}{\n{\tau}^\two}\right],\hskip2cm
\end{eqnarray}
with
\begin{eqnarray}\label{estm13-6}
{\mathcal{X} }_\delta=\C_\twelve\bigl( \n{\uu_\one} +\n{\uu_\two}+\n{\uu_\one}_\three\bigl)+\frac{(\C_\twelve)^\two}{2\delta}\Big(\n{\uu_\one}^\three_\two +\n{\sigma_\one}^\two_\two+2\n{\sigma_\two}^\two_\two+\n{\tau_\one}^\two_\two\Big).
\end{eqnarray}
The function ${\mathcal{X} }_\delta$, defined in (\ref{estm13-6}), is positive. Moreover, because of the class
of solutions we consider, ${\mathcal{X} }_\delta$ belongs to $L^\one(0, T)$. Therefore, choosing $\delta > 0$ small enough, we
deduce from Gronwall's lemma that $\uu = 0$, $\sigma = 0$ and $\tau=0$ in $\Q_\T$, and that consequently $\uu_\one =\uu_\two$, $\sigma_\one = \sigma_\two$, $\tau_\one=\tau_\two$ in $\Q_\T$ and the system (\ref{e13-1}) has a unique solution.


\begin{thebibliography}{100}
\bibitem{ADN} {\sc{S.Agmon, A.Douglis, L.Nirenberg}}, \emph{Estimates near
the boundary for solutions of elliptic partial differential
equations satisfying general boundary conditions} I,II, Com. Pure
and Appl. Math. \textbf{XII}, pp. 623-727, 1959, \textbf{XVII}, pp.
35-92, 1964.
\bibitem{BV3} {\sc{H.Beir\~{a}o Da Veiga}}, \emph{Diffusion on viscous fluids. Existence and asymptotic properties of solutions,}
Annali della Scuola Normale Superiore di Pisa, classe di scienze
4{e} série \textbf{10}, No. 2, pp. 341-355, 1983.
\bibitem{GS} {\sc{C.Guillopé}} and {\sc{J.C.Saut}}, \emph{Existence
results for the flow of viscoelastic fluids with a differential
constitutive law,} Nonlinear Analysis Methods \& Applications
\textbf{15}, Number 9, pp. 849-869, 1990.
\bibitem{ZS} {\sc{Z.Salloum}}, \emph{'Ecoulements de fluides viscoélastiques dans des domaines singuliers}, \'Editions universitaires européennes (EUE), ISBN 978-613-1-51430-2, Sarrebruck, Allemagne, 2010.
\bibitem{S} {\sc{J.Simon}}, \emph{Compact sets in the space $L^{p}(0,T;B)$}, Ann. Mat. Pura Appl.
 \textbf{146}(4), pp. 65-96, 1987.
\bibitem{Ta2} {\sc{R.Talhouk}}, \emph{Analyse mathématique de quelques écoulements de fluides viscoélastiques},
thèse de l'université Paris-Sud, Centre d'Orsay 1994, no. d'ordre
3294.
\bibitem{Te} {\sc{R.Temam}}, \emph{Navier-Stokes
Equations. Theory and Numerical Analysis}, North-Holland, 1984.
\bibitem{V} {\sc{A.Valli}}, \emph{Periodic and stationary solutions for
compressible Navier-Stokes equations via a stability method}, Annali
della Scuola Normale Superiore di Pisa, classe di scienze 4{e}
série \textbf{10}, No. 4, pp. 607-647, 1983.
\end{thebibliography}
\end{document}